\pdfoutput=1                      

\documentclass[10pt,a4paper]{article}

\usepackage[accsupp]{axessibility}    

\usepackage{latexsym,amsfonts,amsmath,amsthm,amssymb, graphicx, mathrsfs, fancyhdr,
amscd,
fontenc,verbatim,stix,amsfonts,amsthm,indentfirst}

\numberwithin{equation}{section} 
\usepackage[all]{xy}
\usepackage{braket}

\usepackage{color,xcolor}

\usepackage[draft=false,setpagesize=false,pdfstartview=FitH,
colorlinks=true,citecolor=blue,pagebackref=true]{hyperref}

\theoremstyle{definition}
\newtheorem{definition}{Definition}[section]
\newtheorem{example}[definition]{Example}

\theoremstyle{plain}
\newtheorem{proposition}[definition]{Proposition}
\newtheorem{lemma}[definition]{Lemma}
\newtheorem{theorem}[definition]{Theorem}
\newtheorem{corollary}[definition]{Corollary}

\newtheorem*{conjecture*}{Conjecture}

\theoremstyle{definition}
\newtheorem{remark}[definition]{Remark}

\newcommand{\diver}{\mathrm{div}}

\newcommand{\cut}{\mathrm{cut}}
\newcommand{\ind}{\mathrm{ind}}

\newcommand{\eps}{\varepsilon}

\newcommand{\di}{\mathrm{d}}

\newcommand{\R}{\mathbb R}

\newcommand{\lip}{\mathrm{Lip}}
\newcommand{\loc}{\mathrm{loc}}
\newcommand{\disp}{\displaystyle}

\newcommand{\gr}{\mathcal{G}}

\newcommand{\Ric}{\mathrm{Ric}}
\newcommand{\Sec}{\mathrm{Sec}}

\newcommand{\LL}{\mathscr{L}}

\newcommand{\ha}{h}
\newcommand{\UU}{U}

\makeatletter      
\newcommand*\owedge{\mathpalette\@owedge\relax}
\newcommand*\@owedge[1]{
	\mathbin{
		\ooalign{
			$#1\m@th\bigcirc$\cr
			\hidewidth$#1\m@th\wedge$\hidewidth\cr
		}
	}
}
\makeatother

\allowdisplaybreaks[1]

\makeatletter
\def\blfootnote{\gdef\@thefnmark{}\@footnotetext}
\makeatother

\begin{document}

\author{Xavier Cabr\'e\textsuperscript{1,2,3}, Giovanni Catino\textsuperscript{4}, Luciano Mari\textsuperscript{5}, \\ Paolo Mastrolia\textsuperscript{5}, Alberto Roncoroni\textsuperscript{4}}

\title{\textbf{Gradient estimates for the Green kernel under spectral Ricci bounds, and the stable Bernstein theorem in $\R^4$}}

\date{}
\maketitle

\blfootnote{\emph{Grants and funding:} X.C. is supported by the Spanish grants PID2021-123903NB-I00 and RED2022-134784-T funded  through MCIN/AEI/10.13039/501100011033 and ERDF “A way of making Europe”, by the Catalan grant 2021-SGR-00087, and by the Spanish State Research Agency through the Severo Ochoa and Mar\'{\i}a de Maeztu Program for Centers and Units of Excellence in R\&D (CEX2020-001084-M). G.C., P.M. and A.R. are members of GNSAGA (Gruppo Nazionale per le Strutture Algebriche, Geometriche e loro Applicazioni). G.C., L.M., P.M. and A.R. are partially supported by the 2022 PRIN project no. 20225J97H5 (Italy) ``Differential-geometric aspects of manifolds via Global Analysis''.}

\begin{center}
	\textit{To the memory of Zindine Djadli}
\end{center}

\smallskip

\newcommand{\Addresses}{{
  \bigskip
  \footnotesize
  
  \textsuperscript{1} Xavier Cabr\'e, \textsc{ICREA, Pg. Lluis Companys 23, 08010 Barcelona, Spain.}\par\nopagebreak
  
  \textsuperscript{2} Xavier Cabr\'e, \textsc{Universitat Polit\`ecnica de Catalunya, Departament de Matem\`{a}tiques and IMTech, Av. Diagonal 647, 08028 Barcelona, Spain.}\par\nopagebreak
  
  \textsuperscript{3} Xavier Cabr\'e, \textsc{Centre de Recerca Matem\`atica, Edifici C, Campus Bellaterra, 08193 Bellaterra, Spain.}\par\nopagebreak
  \textit{E-mail address}: \texttt{xavier.cabre@upc.edu}

  \bigskip

  \textsuperscript{4} G.~Catino \& A.~Roncoroni, \textsc{Dipartimento di Matematica, Politecnico di Milano, Piazza Leonardo da Vinci 32, 20133, Milano, Italy.}\par\nopagebreak
  \textit{E-mail addresses}: \texttt{giovanni.catino@polimit.it}, \texttt{alberto.roncoroni@polimit.it}

  \bigskip

  \textsuperscript{5} L.~Mari \& P.~Mastrolia, \textsc{Dipartimento di Matematica, Universit\`a degli Studi di Milano, Via Cesare Saldini 50, 20133 Milano, Italy}\par\nopagebreak
  \textit{E-mail addresses}: \texttt{luciano.mari@unimi.it}, \texttt{paolo.mastrolia@unimi.it}

}}

\normalsize

\begin{abstract}
We describe a method to prove new integral inequalities for stable minimal hypersurfaces in Euclidean space. As an application, we give a simple proof that complete, two sided, stable minimal hypersurfaces in $\R^4$ are hyperplanes. A core part of the argument hinges on the fact that stable minimal hypersurfaces in non-negatively curved spaces are examples of manifolds with a spectral Ricci curvature lower bound; in particular, we prove a sharp pointwise gradient estimate for the Green kernel on non-parabolic manifolds with spectral Ricci lower bounds, extending previous work by Colding.
%

\end{abstract}

\





\section{Introduction}

Over the past four years, one of the main open questions in minimal hypersurface geometry, the stable Bernstein problem about $2$-sided, complete stable minimally immersed hypersurfaces $M^n \to \R^{n+1}$ (Problem 102 in \cite{yau_problem}), found a definitive answer except for dimension $n=6$:

\begin{theorem}[\textbf{Stable Bernstein Theorem}, \cite{docarmopeng1979,fischer_schoen,pogor,CL,cl_2,CMR,clms,mazet}]\label{teo_stableber}
	Let $M^n \to \R^{n+1}$ be a connected, $2$-sided, complete stable minimally immersed hypersurface. If $n \le 5$, then $M$ is a hyperplane.
\end{theorem}

Recall that a hypersurface $M^n \to \R^{n+1}$ is said to be $2$-sided if there exists a globally defined normal vector field along $M$, and that a $2$-sided minimal hypersurface $M^n \to \R^{n+1}$ is stable if its second variation operator $-\Delta - |A|^2$ is non-negative in the spectral sense, i.e. if
\begin{equation}\label{basicstab}
 \int_M |A|^2 \psi^2 \leq \int_M|\nabla \psi|^2  \qquad \forall \, \psi \in \lip_c(M)\,,
\end{equation}
where $\lip_c(M)$ denotes the set of Lipschitz functions with compact support in $M$ and $A$ is the second fundamental form of $M$.

Theorem \ref{teo_stableber} has a long history. When $M$ is the boundary of a set which is locally minimizing for the perimeter, the conclusion was proved for each $n\leq 6$ as a consequence of the classification of minimizing cones in $\R^{n+1}$ due to Fleming \cite{Fle}, De Giorgi \cite{DeG}, Almgren \cite{Alm} and Simons \cite{Sim}, in connection to the solution of the graphical Bernstein problem; see \cite{Giusti_book, maggi} for a thorough account. Counterexamples to the validity of Theorem \ref{teo_stableber} when $n\geq 7$ were constructed by Hardt \& Simon \cite{simhar} and, for $M$ an entire graph,  when  $n\geq 8$ by Bombieri, De Giorgi \& Giusti \cite{BDGG}. These counterexamples are built upon {\em Simons' cone}
\[
C^7 = \big\{ (x,y) \in \R^4 \times \R^4 : |x|=|y| \big\} \subset \R^8,
\]
a singular minimal hypersurface which is shown by Simons \cite{Sim} to be stable, and by Bombieri-de Giorgi and Giusti \cite{BDGG} to be perimeter minimizer (see also \cite{dephi_pao} for a simplified proof). 

Compared to the case of minimizing boundaries, when $M$ is just a minimally immersed stable hypersurface new difficulties arise and require further, different, tools. In dimension $n=2$, Theorem \ref{teo_stableber} was settled in 1979-81 through the independent works of Do Carmo \& Peng \cite{docarmopeng1979}, Fischer-Colbrie \& Schoen \cite{fischer_schoen} and Pogorelov \cite{pogor}. In the same years, Schoen, Simon \& Yau \cite{SSY} established the result in dimensions $n \in \{3,4,5\}$ under the additional condition that $M$ has Euclidean volume growth, that is, geodesic balls $B_r$ of $M^n$ centered at a fixed origin satisfy
\begin{equation}\label{eq_intrivolume}
|B_r| \le C_nr^n \qquad \text{for some constant $C_n>0$.}
\end{equation}
The core of the argument in \cite{SSY} (see also \cite{CL} for $n=3$) is the derivation, in dimension $n \le 5$, of the inequality
\begin{equation}\label{eq_SSY}
\int_M |A|^p |\psi|^p \le C_{n,p} \int_M |\nabla \psi|^p \qquad \forall \, \psi \in \lip_c(M),
\end{equation}
for a positive constant $C_{n,p}$ and some $p \ge n$. Assuming \eqref{eq_intrivolume} one may then apply a standard cutoff argument (if $p>n$) or a logarithmic cutoff one (if $p=n$) to conclude that $|A| \equiv 0$ (we refer e.g. to \cite[Theorems 2.21 and 2.22]{colding_minicozzi_book} for further details). Inequality \eqref{eq_SSY} relies on \emph{Simons' equation}, which in ambient Euclidean space reads as
\begin{equation}\label{EqSimons}
	\frac{1}{2}\Delta|A|^2 = |\nabla A|^2 - |A|^4,
\end{equation}
see e.g. \cite[Formula (2.16)]{colding_minicozzi_book}.

For $n\geq 3$, the validity of Theorem \ref{teo_stableber} remained open up until 2021, when Chodosh \& Li \cite{CL} obtained the first complete proof in dimension $n=3$. Subsequently, Chodosh \& Li \cite{cl_2} and Catino, Mastrolia \& Roncoroni \cite{CMR} reproved the theorem using new techniques, each relying on some features specific to $3$-dimensional hypersurfaces (the use of Gauss-Bonnet theorem for level sets in \cite{CL,cl_2}, and the completeness of a conformally deformed metric in \cite{CMR}). In \cite{clms}, Chodosh, Li, Minter \& Stryker settled the case $n=4$ by refining the method in \cite{cl_2} and introducing new ideas to bypass the use of Gauss-Bonnet theorem. Among their contributions is a spectral Bishop-Gromov volume comparison theorem, later extended to any dimension by work of Antonelli \& Xu \cite{antonelli_xu}. Finally, Mazet \cite{mazet} sharpened the techniques in \cite{clms} to solve the problem in dimension $n=5$. 

The case $n=6$ remains open in full generality. Under the extrinsic volume bound 
\begin{equation}\label{eq_extrinsic}
	\vert M \cap \mathbb{B}^{\R^{7}}_r(0)\vert \le C r^6,
\end{equation}
a stronger assumption than \eqref{eq_intrivolume} which in particular forces $M$ to be proper in $\R^7$ (see, for instance, \cite[Appendix]{bra_fab_mar_vie}), Schoen \& Simon \cite{ss} proved that $M$ must be a hyperplane when $M$ is embedded, and more recently Bellettini \cite{bel} extended this to immersed hypersurfaces. We also refer the reader to the work of Tysk \cite{t}. 

In light of Schoen-Simon-Yau's result, a proof of Theorem \ref{teo_stableber} in dimensions $n \le 5$ follows once \eqref{eq_intrivolume} is established, which is indeed the striking achievement of  \cite{cl_2,clms,mazet}. 
A major obstacle in dimension $n=6$ is that certain key algebraic inequalities used in the proof of Theorem \ref{teo_stableber} break down. Indeed, as highlighted in a note by Antonelli \& Xu \cite{ax_2}, the estimates in \cite{cl_2,clms,mazet} to get \eqref{eq_intrivolume} cannot be sharpened to cover the case $n=6$ without introducing genuinely new ideas. Furthermore, even establishing that $|B_r| \le C r^6$ would not suffice to conclude Theorem \ref{teo_stableber} unless inequality \eqref{eq_SSY} is extended as well to $n=6$ and some $p \ge 6$, a step that likewise appears to require new ingredients. In this direction, a noteworthy aspect of \cite{bel} is that the restriction $n \le 6$ emerges naturally in certain estimates connecting Simons' equation with the $L^2$-Sobolev inequality (cf. \cite[Theorem 3]{bel}).

The purpose of this paper is to introduce new tools that could be useful to address the Bernstein problem for $M^6 \to \R^7$; as an application, in Section \ref{sec_bernstein} we provide a streamlined proof of Theorem \ref{teo_stableber} for $n=3$. All ingredients used in the proof are purely analytical, with the exception, essentially,   of Simons' equation. 

Our general strategy is to incorporate weights into the  Schoen-Simon-Yau inequality in order to relax the algebraic constraints  relating $p$ and $n$. To this end, one may expect that helpful weight functions should have controlled Laplacian. To take advantage of this, the integration-by-parts argument leading to \eqref{eq_SSY} should be refined to include second order terms both in $|A|$ (to exploit Simons' equation) and in the weight. Effectively combining the resulting terms, however, is nontrivial. In earlier work \cite{CMR}, some of us obtained a version of \eqref{eq_SSY} weighted by powers of a positive function satisfying
\begin{equation}\label{stab_u}
\Delta u + |A|^2 u \le 0 \quad \text{in }  M\, ,
\end{equation}
whose existence follows from stability. However, the proof relied on a rather intricate use of Simons' equation and is limited to dimension $n=3$. 

Here we present a different method, valid in all dimensions, conceptually simpler and amenable to iteration by introducing additional weights. For minimal hypersurfaces in Euclidean space and with $u$ as in \eqref{stab_u}, the method yields (see equation \eqref{EstBer})
\[
\int_M u^{-2t\delta} |A|^p |\psi|^p \le C \int_M u^{-2t\delta} |\nabla \psi|^p \qquad \forall \, \psi \in \lip_c(M) 
\]
for a positive constant $C$ and suitable parameters $t,\delta$ and $p$ depending on $n$. To get $|A| \equiv 0$ one may then seek for a family of cut-off functions $\{\psi_j\}$ converging to a positive limit and for which the right-hand side tends to zero as $j \to \infty$. We construct them in the form  $\psi_j = \eta_j(\gr)$, where $\gr$ is the Green kernel of $-\Delta$ centered at a fixed origin. To proceed and conclude, we thus need to estimate $|\nabla \gr|$ either pointwise or in integral form, as explained next.

The second main result of the present paper is a sharp pointwise estimate for $|\nabla \gr|$. To state it, let us observe that a minimal hypersurface $(M^n,g)$ in Euclidean space, and more generally in an ambient manifold $(N^{n+1}, \bar g)$ whose sectional curvature satisfies $\overline{\Sec} \ge 0$, enjoys the bound 
\begin{equation}\label{eq_lowerric_intro}
\Ric \ge - \frac{n-1}{n}|A|^2 g.
\end{equation}
If $M$ is stable, i.e. if the Jacobi operator $-\Delta - \big[ |A|^2 + \overline{\Ric}(\nu,\nu)\big]$ is non-negative in the spectral sense, then $M$ is an example of manifold whose Ricci tensor satisfies a \emph{spectral Ricci bound}
\[
\Ric \ge - \frac{n-1}{n-2} \tau V g, \qquad -\Delta - V \ge 0.
\] 
for some $\tau \in [0,\infty)$ and $V \in C(M)$. Here, $-\Delta - V \ge 0$ is meant in the spectral sense, i.e. the quadratic form 
\[
Q_V(\psi) \doteq \int_M |\nabla \psi|^2 - V \psi^2
\] 
is non-negative for each $\psi \in \lip_c(M)$. The structure of such manifolds was studied in depth by Li \& Wang \cite{liwang_ens}; see also Cheng \& Zhou \cite{cheng_zhou_1}, Pigola, Rigoli \& Setti \cite{prs_book} and, in more recent years, Antonelli, Pozzetta \& Xu \cite{apx, anto_xu_gluing}, Bour \& Carron \cite{boucar}, Carron \& Rose \cite{carrose}, and \cite{cmmr_criticality} by some of the authors of the present paper. The results in these works show that various geometric and topological properties of a complete manifold can be controlled under spectral Ricci bounds in the range
\[
\tau \in [0,1], 
\]
However, when a positive Green kernel of $-\Delta$ exists (i.e., when $M$ is non-parabolic), its behaviour under spectral Ricci bounds is still poorly understood. To state our result, recall that an operator $-\Delta - V$ is said to be
\emph{subcritical} if $-\Delta - V \ge 0$ and if $-\Delta - V$ admits a positive Green kernel. Moreover, $-\Delta - V$ is said to have \emph{finite index} if its Morse index
\[
\ind_V(M) = \sup \left\{ \dim Z \ : \ Z \, \text{subspace of }\,  C^\infty_c(M), \ \ Q_V(\psi) < 0 \ \ \text{ for each } \, \psi \in Z \backslash \{0\} \right\}
\]
is finite. By works of Fischer-Colbrie \cite{fischercolbrie} and Devyver \cite{devyver} (see also \cite[Thm 2.41]{bmr_mem} for an alternative proof), $\ind_V(M)<\infty$ is equivalent to the existence of a compact set $K$ such that $-\Delta - V \ge 0$ in $M \backslash K$. 

The following is our result. Its statements are contained in Theorems \ref{teo_gradesti} and \ref{teo_gradesti_kernel}.

\begin{theorem}\label{teo_gradesti_intro}
Let $(M^n,g)$ be a complete Riemannian manifold of dimension $n \ge 3$ satisfying
\[
\Ric \ge - \frac{n-1}{n-2} \tau V g, 
\]
for some $0 \le V \in C^{0,\alpha}_\loc(M)$ and 
\[
\tau \in \left[0, \frac{3}{4}\right) \ \ \text{ if } \, n=3, \qquad \tau \in [0,1) \ \ \text{ if } \, n \ge 4.
\]	
Assume that $M$ is non-parabolic, and let $\gr$ be the minimal positive Green kernel of $-\Delta$ on $M$ with pole at $o$. 
\begin{itemize}
	\item[(i)] If $-\Delta - V$ has \emph{finite index}, then for each smooth open set $\UU \Subset M$\footnote{Here, and in what follows, $\UU \Subset M$ means that $\UU$ is compactly contained in $M$, i.e., its closure $\overline{\UU}$ is compact and $\UU \subset M$.} for which there exists a solution $0 < u \in C^2(M \backslash \UU)$ to $\Delta u + Vu \le 0,$ it holds
\[
|\nabla \gr|^{\frac{n-2}{n-1}} \le C u^{\tau} \gr^{1-\tau}\qquad \text{in } \, M \backslash \UU, \qquad \text{where } \, C = \max_{\partial \UU} \left\{ |\nabla \gr|^{\frac{n-2}{n-1}} u^{-\tau} \gr^{\tau-1} \right\}.
\]
\item[(ii)] If $-\Delta - V \ge 0$ is subcritical, letting $u$ be a positive Green kernel for $-\Delta-V$ with pole~$o$, it holds
\begin{equation}\label{eq_gradesti_kernel_intro}
	|\nabla \gr|^{\frac{n-2}{n-1}} \le (n-2) |\mathbb{S}^{n-1}|^{\frac{1}{n-1}} u^{\tau} \gr^{1-\tau} \qquad \text{in } \, M \backslash \{o\}.
\end{equation}
Moreover, equality holds for some $x_0 \in M \backslash \{o\}$ if and only if $V \equiv 0$ and $M$ is isometric to $\R^n$.
\end{itemize}
\end{theorem}

\begin{remark}
Note that setting $V \equiv 0$, $\tau = 0$ and $u = \gr$ in  \textit{(ii)} we recover the sharp estimate 
\[
	|\nabla \gr|^{\frac{n-2}{n-1}} \le (n-2) |\mathbb{S}^{n-1}|^{\frac{1}{n-1}} \gr \qquad \text{in } \, M \backslash \{o\}
\]
for the Green kernel of a non-parabolic manifold with $\Ric \ge 0$, proved by Colding in \cite{colding} with a different method. From \eqref{eq_lowerric_intro}, the case of stable minimal hypersurfaces in non-negatively curved ambient spaces is recovered with the choices 
\[
\tau = \frac{n-2}{n}, \qquad V = |A|^2\, ;
\]
see Corollary \ref{cor_gradesti_minimal}.
\end{remark}

\begin{remark}
	Item {\em (i)} holds, more generally, for harmonic functions defined in $M \backslash \UU$ which have minimal growth at infinity, regardless of the non-parabolicity of the entire $M$. On the other hand, in {\em (ii)}, since $V \ge 0$ the subcriticality of $-\Delta - V$ ensure that of $-\Delta$, so $M$ is automatically non-parabolic. 
\end{remark}

Although our paper focuses on stable minimal hypersurfaces, the techniques presented in Section \ref{sec_spectral} possess general validity, and we believe that it will be possible to implement the method developed in this paper to get flatness results in low dimensions for other elliptic problems.

\vspace{0.3cm}

\noindent\textbf{Structure of the paper.} Section \ref{sec_bernstein} contains a complete proof of the Stable Bernstein Theorem in $\R^4$. The proof is self-contained with exception of a simple computation, deferred to Section \ref{sec_spectral}, which establishes our key Lemma \ref{prop_magic}.
In Section \ref{sec_spectral}, we first prove and comment on Lemma \ref{prop_magic}, which yields a method to combine sub- and supersolutions of linear differential equations into new subsolutions, and then discuss its applications to the weighted Schoen-Simon-Yau inequality (Corollary \ref{cor_magic_minimal_2}) and to pointwise gradient estimates for minimal harmonic functions -- in particular, for the Green kernel -- on manifolds with spectral Ricci bounds; see Theorems \ref{teo_gradesti} and \ref{teo_gradesti_kernel}. Finally, in the Appendix we prove a technical proposition, used on various occasions throughout the rest of the article.

\section{A new proof of the stable Bernstein theorem in $\R^4$}\label{sec_bernstein}

We present a proof of the stable Bernstein theorem, Theorem \ref{teo_stableber}, in case the dimension is $n=3$. 
Since most of the steps hold in general dimensions, we shall restrict to $n=3$ only at the very end of the argument.

\begin{proof}[Proof of Theorem \ref{teo_stableber} for $n=3$]
Let $n\geq 3$ and $M^n\rightarrow\R^{n+1}$ be a connected, $2$-sided, complete stable minimal hypersurface. By a point-picking, scaling, and limiting argument (see \cite[Lecture 3]{white} and also \cite[beginning of Section 5]{CL} or \cite[Subsection 8.1]{chodosh_LN}), we can assume without loss of generality that 
\[
|A| \in L^\infty(M). 
\]
It is known that every minimal hypersurface admits an isoperimetric inequality. This is the Allard \cite{All} and Michael-Simon \cite{micsim} isoperimetric inequality generalized by Hoffman \& Spruck \cite{HS}; see also \cite[Section 2]{CM} for a simpler quick proof. Since $n>2$, as a consequence we deduce that $M$ is non-parabolic, i.e. $M$ admits a minimal, positive Green kernel $\gr$ centered at a fixed pole $o$, i.e. a minimal positive solution to $\Delta \gr = -\delta_o$; see e.g. \cite[Theorem 8.2]{Grygorian_book} and the lecture notes \cite[Corollary 3.12]{Li}. Moreover, by \cite{davies} (see also \cite[Remark 2.4]{ni} and \cite[Theorem 2.8]{car}), $\gr$ vanishes at infinity. In particular, for $t>0$, the sets $\{\gr \ge t\}$ are compact and, by Sard's theorem, their boundaries are smooth for almost every $t>0$. By integrating $-\Delta \gr = \delta_o$ against suitable smooth approximations of the indicatrix of $\{\gr \ge t\}$, one sees that  
\begin{equation}\label{int_nablag}
	\int_{\{\gr=t\}}|\nabla\gr|=1\qquad \text{for a.e. } \, t>0 
\end{equation}
(indeed,  identity \eqref{int_nablag} holds regardless of the compactness of $\{\gr \ge t\}$, see \cite[Proposition 2.3]{mrsv}).

Choose a $C^2$ function $u > 0$ satisfying 
$$
\Delta u + |A|^2 u \le 0 \quad \text{on }  M\, ,
$$
which is granted by stability (see \cite{fischer_schoen} and also \cite[Lemma 3.10]{prs_book} or \cite[Proposition 1.39]{colding_minicozzi_book}; these results also guarantee the existence of $u>0$ solving the above with the equality sign,  
but the inequality suffices for our argument). 

Recall that the Ricci curvature of $M$ satisfies
\begin{equation}\label{eq_Ricbound}
	\Ric = - A^2 \ge - \frac{n-1}{n}|A|^2 g\, ;
\end{equation}
the equality follows from traced Gauss equation (see e.g. \cite[equation (2.5)]{colding_minicozzi_book}), while the inequality follows from the algebraic Kato inequality for traceless symmetric tensors (see for instance \cite[Lemma 2.4]{huisken}). By Bochner's formula and again the algebraic Kato inequality (see e.g.  \cite[Chapter 1 - Section 3]{SY_book}), $|\nabla \gr|^2$ thus satisfies  
\begin{equation}\label{eq_nablaG}
	\begin{array}{lcl}
	\disp \frac{1}{2}\Delta |\nabla \gr|^2 & \ge & \disp \frac{n}{n-1} |\nabla |\nabla \gr||^2 + \Ric(\nabla \gr, \nabla \gr) \\[0.4cm]
	& \ge & \disp \frac{n}{n-1} |\nabla |\nabla \gr||^2 - \frac{n-1}{n}|A|^2 |\nabla \gr|^2\, .
	\end{array}
\end{equation}

Hence, the functions
$$
w\doteq|\nabla \gr|^{\frac{n-2}{n-1}}
$$
and $u$ satisfy
\[
\left\{ \begin{array}{l}
\Delta w+\frac{n-2}{n}|A|^2w\geq 0 \\[0.2cm]
\Delta u +|A|^2u\leq 0 
\end{array}\right. \qquad \text{in } \, \{|\nabla \gr|>0\}. 	
\]
Therefore, by Lemma \ref{prop_magic} $(i)$ in Section $3$, the function
\[
\xi \doteq w^{\frac{n}{2}} u^{- \frac{n-2}{2}}=|\nabla \gr|^{\frac{n-2}{n-1} \cdot \frac{n}{2}} u^{- \frac{n-2}{2}}
\]
satisfies
\[
\Delta \xi \ge \frac{n(n-2)}{4} \xi \left| \nabla \log \left( \frac{w}{u} \right)\right|^2 \ge 0 \qquad \text{pointwise in } \, \{|\nabla \gr|> 0\}.
\] 

Next, by \eqref{eq_Ricbound} and $|A| \in L^\infty(M)$ we deduce that the Ricci curvature of $M$ is bounded below by some constant. Hence, we can use Cheng-Yau's estimate (see \cite[Theorem 6]{ChengYau}) in $M \backslash \UU$, where 
\[
\UU \doteq \{\gr>1\}\, , 
\]
using balls all of the same radius, to conclude that $|\nabla \gr| \le C_1 \gr$ in $M \backslash \UU$ for some constant $C_1>0$. Moreover, taking $C_2=\max_{\partial \UU}u^{-1}$, we have $\gr \le C_2 u$ in $M \backslash \UU$. To see this, simply note that $C_2u-\gr$ is nonnegative on $\partial \UU$, has a nonnegative $\liminf$ as $x\rightarrow\infty$, and it cannot achieve an interior (negative) minimum, since it is superharmonic in $M\backslash\UU$. It follows that  
\[
\xi(x) \le C_3 \gr(x)^{\frac{n-2}{n-1} \cdot \frac{n}{2}} u(x)^{-\frac{n-2}{2}} \le C_4 \gr(x)^{\frac{n-2}{n-1} \left[ \frac{n}{2} - \frac{n-1}{2}\right]} \to 0 \qquad \text{as } \, x \to \infty.
\]
We can therefore use the maximum principle exactly as before to the superharmonic function 
\[
C_5 \gr-\xi \qquad \text{ in } \{|\nabla \gr|> 0\}\backslash \UU\, , 
\] 
with $C_5=\max_{\partial \UU}\xi$. From
\[
\begin{array}{ll}
C_5 \gr-\xi\geq 0 & \quad \text{on } \, \partial\left(\{|\nabla \gr|> 0\}\backslash \UU\right) \ \cup \ \{|\nabla \gr|=0\}, \\[0.2cm]
C_5 \gr(x)-\xi(x)\to 0 & \quad \text{as } \, x\to\infty. 
\end{array}
\]
We deduce that $\xi \le C_5 \gr$ in $M \backslash \UU$. Rearranging, we obtained the following pointwise gradient estimate:
\begin{equation}\label{grad_est_3_sec2}
	|\nabla \gr|^{\frac{n-2}{n-1}} \le C u^{\frac{n-2}{n}} \gr^{\frac{2}{n}} \qquad \text{in } \, M \backslash \UU \, ,
\end{equation}
for some constant $C>0$.

Next, by Simons' equation \eqref{EqSimons} and the refined Kato inequality 
\[
|\nabla A|^2 \ge \frac{n+2}{n}| \nabla |A||^2
\]
due to \cite{SSY} (see also \cite[Lemma 2.1]{colding_minicozzi_book}), it holds
\begin{equation}\label{eq_inesimons}
	|A|\Delta |A| \ge \frac{2}{n}|\nabla |A||^2 - |A|^4 \qquad \text{in } \, \{ |A|>0 \}.
\end{equation}
Thus, a computation shows that, for each $t \in [0,1]$, the functions
\[
\tilde w \doteq |A|^{\frac{n-2}{n}}\qquad\text{and}\qquad \tilde u \doteq u^t
\]
satisfy
\[
\left\{ \begin{array}{l}
\Delta \tilde w+\frac{n-2}{n}|A|^2 \tilde w\geq 0 \\[0.2cm]
\Delta \tilde u +t |A|^2\tilde u\leq 0 
\end{array}\right. \qquad \text{in } \, \{|A|>0\}. 	
\]
Hence, for $\delta>0$, a further application of Lemma \ref{prop_magic} $(i)$ below guarantees that the function
\[
z \doteq |A|^{\frac{n-2}{n}(1+\delta)} u^{-t\delta}
\]
satisfies
\[
\Delta z + |A|^2 z \ge \left( \left[t - \frac{n-2}{n}\right] \delta + \frac{2}{n} \right)|A|^2 z \qquad \text{pointwise in } \, \{|A|>0\}.
\]
As a consequence, by using Proposition \ref{prop_weakly} in the Appendix, one sees that $z \in H^1_\loc(M)$ and that the previous inequality holds weakly in all of $M$.

Given $z_0 \in H^1_\loc(M)$ and $\psi \in \lip_c(M)$, choose  $\xi=z_0\psi$ in the stability condition \eqref{basicstab} (which by density holds for every test function in $H^1(M)$ with compact support). Integration by parts shows that
\[
\int_M z_0 (\Delta z_0 + |A|^2 z_0) \psi^2 \le \int_M z_0^2 |\nabla \psi|^2
\]
holds in the weak sense in $M$. In particular, plugging $z_0 = z$ we get
\[
\left( \left[t - \frac{n-2}{n}\right] \delta + \frac{2}{n} \right) \int_M \left(|A|^{1+ \frac{n-2}{n}(1+\delta)} u^{-t\delta}\right)^2 \psi^2 \le \int_M \left(|A|^{\frac{n-2}{n}(1+\delta)} u^{-t\delta}\right)^2 |\nabla \psi|^2.
\]
Thus, if 
\[
\left[ t - \frac{n-2}{n}\right] \delta + \frac{2}{n} > 0,
\]
by taking $\psi=\varphi^{1+\frac{n-2}{n}(1+\delta)}$ for some $\varphi\in \lip_c(M)$ and using Young's inequality we deduce, for some constant $C$ depending only on $n,\delta$ and $t$ whose value may vary from line to line,
\begin{align*}
\int_M &\left(|A|^{1+ \frac{n-2}{n}(1+\delta)} u^{-t\delta}\right)^2 |\varphi|^{2+ 2\frac{n-2}{n}(1+\delta)} \le C \int_M |A|^{2\frac{n-2}{n}(1+\delta)} u^{-2t\delta} \varphi^{2\frac{n-2}{n}(1+\delta)}|\nabla\varphi|^2\\
&\le \frac{1}{2} \int_M \left(|A|^{1+ \frac{n-2}{n}(1+\delta)} u^{-t\delta}\right)^2 |\varphi|^{2+ 2\frac{n-2}{n}(1+\delta)} + C \int_M u^{- 2t\delta} |\nabla \varphi|^{2+ 2\frac{n-2}{n}(1+\delta)}.
\end{align*}
Therefore
\begin{equation}\label{EstBer}
\int_M \left(|A|^{1+ \frac{n-2}{n}(1+\delta)} u^{-t\delta}\right)^2 |\varphi|^{2+ 2\frac{n-2}{n}(1+\delta)}  \leq C \int_M u^{- 2t\delta} |\nabla \varphi|^{2+ 2\frac{n-2}{n}(1+\delta)}.
\end{equation}

We now choose $\varphi=\eta(\gr)$, where $0 \le \eta\in C^1_c((0,1])$ is nondecreasing and $\eta\equiv 1$ in $[1,\infty)$. Note that $\varphi=\eta(\gr)$ will have compact support and that will be identically $1$ is a neighborhood of the pole $o$ of $\gr$; hence, $\varphi\in\lip_c(M)$. Then, from the coarea formula
\begin{equation}\label{eq_coareat}
	\begin{array}{l}
		\disp \int_M \left(|A|^{1+ \frac{n-2}{n}(1+\delta)} u^{-t\delta}\right)^2\eta(\gr)^{2+ 2\frac{n-2}{n}(1+\delta)} \le \disp C \int_M u^{- 2t\delta} \eta'(\gr)^{2+ 2\frac{n-2}{n}(1+\delta)}|\nabla \gr|^{2+ 2\frac{n-2}{n}(1+\delta)} \\[0.5cm]
		\qquad \qquad \qquad \qquad \le \ \ \disp C \disp \int_0^1 \eta'(s)^{2+ 2\frac{n-2}{n}(1+\delta)} \left[ \int_{\Sigma_s} u^{- 2t\delta}|\nabla \gr|^{1+ 2\frac{n-2}{n}(1+\delta)} \right] \di s,
	\end{array}
\end{equation}
where 
\[
\Sigma_s \doteq \{\gr=s\}. 
\]
From \eqref{grad_est_3_sec2} we have 
\begin{equation}\label{esti_pointt}
	u^{-t\delta}|\nabla \gr|^{\frac{n-2}{n}(1+\delta)} \le C u^{-t\delta +  \frac{n-1}{n}(1+\delta)\frac{n-2}{n}} \gr^{\frac{n-1}{n}(1+\delta)\frac{2}{n}}.
\end{equation}
We  choose $t \in (0,1]$ and $\delta \ge 0$ satisfying  the following conditions:
\begin{equation}\label{eq_tt}
	\begin{cases}
		\disp 4\frac{n-1}{n^2}(1+\delta)  \ge  \disp 1 + 2\frac{n-2}{n}(1+\delta) \\[0.2cm]
		t \delta  =  \disp \frac{n-1}{n}(1+\delta)\frac{n-2}{n} \\[0.2cm]
		\disp \left[t - \frac{n-2}{n}\right] \delta + \frac{2}{n}  >  0. 
	\end{cases}
\end{equation}
The first inequality has a non-negative solution $\delta$ if and only if $n=3$, which accounts for the dimension restriction. For $n=3$, the three conditions become
\[
\delta \ge \frac{7}{2}, \qquad t = \frac{2}{9}\frac{1+\delta}{\delta}, \qquad \left[\frac{1}{3}-t\right]\delta < \frac{2}{3}. 
\]
The choice $\delta = 7/2$ and $t = 2/7$ satisfies all of them. 

Once a pair $(\delta, t)$ satisfying \eqref{eq_tt} is found, by the second condition in \eqref{eq_tt}, inequality \eqref{esti_pointt} becomes  
$$
u^{-t\delta}|\nabla \gr|^{\frac{n-2}{n}(1+\delta)} \le C \gr^{\frac{n-1}{n}(1+\delta)\frac{2}{n}}.
$$
Inserting this into \eqref{eq_coareat}, by the first condition in \eqref{eq_tt} and by  \eqref{int_nablag} we deduce that
\[
\begin{array}{lcl}
	\disp \int_M \left(|A|^{1+ \frac{n-2}{n}(1+\delta)} u^{-t\delta}\right)^2\eta(\gr)^{2+ 2\frac{n-2}{n}(1+\delta)} & \le & \disp C \disp \int_0^1 \eta'(s)^{2+ 2\frac{n-2}{n}(1+\delta)} s^{4\frac{n-1}{n^2}(1+\delta)} \left[ \int_{\Sigma_s} |\nabla \gr|\right] \di s \\[0.5cm]
	& \le & C \disp \int_0^1 \eta'(s)^{2+ 2\frac{n-2}{n}(1+\delta)} s^{4\frac{n-1}{n^2}(1+\delta)} \di s\\[0.5cm]
	&\le& \disp C \int_0^1 \eta'(s)^{1+p} s^{p} \di s,
\end{array}
\]
where $p \doteq  1 + 2\frac{n-2}{n}(1+\delta)>0$. For $R>1$ we finally choose
$$
\eta(s)=\begin{cases}
	0 & \text{ if } s\in\left[0,R^{-2}\right), \\
	2+\frac{\log s}{\log R} & \text{ if } s\in\left[R^{-2},R^{-1}\right), \\
	1 & \text{ if } s\ge R^{-1},
\end{cases}
$$ 
to obtain 
\[
\int_{\{ \gr > R^{-1}\}} 
\left(|A|^{1+ \frac{n-2}{n}(1+\delta)} u^{-t\delta}\right)^2 \le \frac{C}{(\log R)^{1+p}} \int_{R^{-2}}^{R^{-1}}s^{-1} = \frac{C}{(\log R)^{p}}  \, . 
\] 
Letting $R\rightarrow\infty$, being $p>0$ we conclude 
$$
|A|\equiv 0 \quad \text{ on } M. 
$$ 
\end{proof}
\subsection*{Final remarks}
Neither the assumption $|A| \in L^\infty(M)$ nor the vanishing of $\gr$ at infinity is actually necessary to establish the gradient estimate \eqref{grad_est_3_sec2} for $\gr$. Indeed, this inequality holds for every connected, complete, $2$-sided stable minimal hypersurface $M^n \to N^{n+1}$ provided $N$ satisfies $\overline{\Sec} \ge 0$ and $\gr$ exists on $M$; see Corollary \ref{cor_gradesti_minimal} below. A natural question is whether, in this more general setting, $M$ must be totally geodesic if $n \ge 3$ is small enough, or at least if $n=3$. To the best of our knowledge, no counterexample has been found so far. A positive answer is straightforward if $\gr$ does not exist on $M$. Indeed, by a standard characterization (see \cite[Theorem 5.1]{Grygorian_book}), in this case $M$ is parabolic, i.e. positive superharmonic functions on $M$ are constant. The existence of a positive solution to
\[
\Delta u + \big[ |A|^2 + \overline{\Ric}(\nu,\nu)\big]u \le 0,
\] 
granted by the stability condition, therefore forces both
\begin{equation}\label{eq_forced}
|A|^2 \equiv 0 \qquad \text{and} \qquad \overline{\Ric}(\nu,\nu) \equiv 0.
\end{equation}
However, if $M$ is non-parabolic things are much more subtle. For instance, the counterexample in \cite[Example 1.2]{cls} shows that the second property in \eqref{eq_forced} may fail. Remarkably, \eqref{eq_forced} does hold in dimension $n=3$ if $N$, beyond the above assumptions, has uniformly positive scalar curvature and weakly bounded geometry (for instance, if it has globally bounded sectional curvature); see \cite{cls} for more details. Apart from this result, techniques suitable for treating the problem in general ambient spaces with non‑negative sectional curvature are rather scarce in the literature. In this respect, we note that the above proof of the stable Bernstein theorem relies only on two facts:
\begin{itemize}
	\item[$a)$] the simpler form of Simons' equation in $\mathbb{R}^{n+1}$ compared to more general ambient spaces, where the additional terms in Simons' equation cannot be absorbed into our integral estimates;
	\item[$b)$] the vanishing of $\gr$ at infinity.  
\end{itemize}
In particular, we do not use properties of the distance function $r(x) = |x|$ in $\R^{n+1}$, especially, the fact that the ambient Hessian of $r^2$ is (at least) twice the identity. One may question this claim by observing that, in the arguments by \cite{All,micsim,HS}, such property is needed to deduce the validity of the isoperimetric inequality on $M \to \R^{n+1}$, which by \cite{davies,ni,car} implies the vanishing of $\gr$ at infinity. However, Brendle \cite{brendle} showed that any minimal hypersurface $M^n \to N^{n+1}$ in a manifold with $\overline{\Sec} \ge 0$ satisfies the isoperimetric inequality provided that balls $B_r^N \subset N$ centered at a fixed origin have maximal volume growth:
\begin{equation}\label{eq_bound_maxvol}
\lim_{r \to \infty} \frac{|B_r^N|}{r^{n+1}} > 0.
\end{equation}
In these ambient spaces $N$, the above Hessian lower bound on $r^2$ generally fails.\par
Summarizing, in our proof, fact $a)$ is the only property forcing to restrict to ambient $\R^4$ instead of general $4$-manifolds with non-negative sectional curvature and satisfying \eqref{eq_bound_maxvol}. Regarding fact $b)$, it might be conceivable that the integral estimates in the above proof could be refined so as to reach the conclusion without knowing, \emph{a-priori}, the validity of $b)$. If this were the case, it would help to address ambient spaces where \eqref{eq_bound_maxvol} fails.

\section{Manifolds with spectral Ricci bounds}\label{sec_spectral}

Let $(M^n,g)$ be a complete Riemannian manifold, and let $V \in L^\infty_\loc(M)$. An operator $-\Delta - V$ is said to be non-negative in a domain $\Omega \subset M$ if the associated quadratic form 
\[
Q_V(\psi) = \int_M|\nabla \psi|^2 - V \psi^2
\]
is non-negative for each $\psi \in \lip_c(\Omega)$ (equivalently, $\psi \in H^1_c(M)$, the set of $H^1$ functions which are zero a.e. away from some compact set). It is known by \cite{fischer_schoen} (see also \cite[Lemma 3.10]{prs_book} in the present generality) that $-\Delta - V \ge 0$ is equivalent to any of the following: 
\begin{itemize}
	\item[(i)] The existence of $0 < u \in H^1_\loc(\Omega)$ solving $\Delta u + Vu \le 0$ weakly in $\Omega$; 
	\item[(ii)] The existence of $0 < u \in C^{1,\alpha}_\loc(\Omega)$ solving $\Delta u + Vu = 0$ weakly in $\Omega$;
\end{itemize}

The operator $-\Delta - V \ge 0$ is said to be \emph{subcritical} if it admits a positive Green kernel with pole at a given origin $o$, i.e. a positive solution to
\[
\Delta u + Vu = -\delta_o,
\]
where $\delta_o$ is the Dirac delta-function at $o$, and \emph{critical} otherwise. For $V \equiv 0$, this dichotomy reproduces the standard one between parabolic and non-parabolic manifolds, corresponding respectively to manifolds whose Laplacian $-\Delta$ is critical and subcritical. Various characterizations of the subcriticality of $-\Delta - V$, and a thorough treatment of the role played by the potential $V$, were given by Murata \cite[Section 2]{murata} and Pinchover \& Tintarev \cite{pinchover, pinchover2, pincho_tinta}. We also suggest the reader to consult \cite{bmr,car}, \cite[Section 2]{cmmr_criticality} and the reference therein for more information.

\begin{remark}
	If $-\Delta -V$ is subcritical, one can consider the minimal positive Green kernel $u$, obtained as the locally uniform limit of solutions to 
	\[
	\left\{ \begin{array}{ll}
		\Delta u_j + Vu_j = -\delta_o & \quad \text{in } \, \Omega_j, \\[0.2cm]
		u_j =0 & \quad \text{on } \, \partial \Omega_j,
	\end{array}\right.
	\]
	for an increasing exhaustion $\{\Omega_j\}$ of $M$ by smooth, relatively compact open sets. Condition $-\Delta - V \ge 0$ guarantees that $u$ is independent of the chosen exhaustion and hence uniquely defined.
\end{remark}

Next, recall that an operator $-\Delta - V \ge 0$ satisfies the following useful property: for each $z \in H^1_\loc(M)$, the inequality 
\begin{equation}\label{eq_id_basic_impo}
\int_M z(\Delta z + V z) \psi^2 \le \int_M z^2 |\nabla \psi|^2
\end{equation}
holds weakly for each $\psi \in \lip_c(\Omega)$, that is,
\[
\int_M V z^2 \psi^2 - \int_M \langle \nabla z, \nabla(\psi^2 z) \rangle  \le \int_M z^2 |\nabla \psi|^2.
\]
Indeed, it is enough to observe that 
\[
\int_M V z^2 \psi^2 - \int_M \langle \nabla z, \nabla(\psi^2 z) \rangle = \int_M V z^2 \psi^2 - \int_M |\nabla(z\psi)|^2 + \int_M z^2 |\nabla \psi|^2
\]
and use that $Q_V(z\psi) \ge 0$.
%
	
%


%
%

\subsection{Construction of subsolutions}

In investigating problems from Riemannian Geometry, it often happens that certain geometrically relevant functions $\psi : M \to \R$ are $C^2$ in $\{\psi > 0\}$ and there they solve
\[
\Delta \psi + a \psi \ge -\kappa \frac{|\nabla \psi|^2}{\psi}
\]
for some $a \in C(M)$ and $\kappa \in (-1,\infty)$, see \cite{prs_book} for a detailed account and examples. The coefficient $\kappa$ typically comes from a refined Kato inequality. As a consequence, $w = \psi^{\kappa+1}$ satisfies
\[
\Delta w + Ww \ge 0 \qquad \text{in } \, \{w>0\},
\]
where we set $W = (1+\kappa )a$. Furthermore, the problem under consideration may naturally lead to the non-negativity of some operator $-\Delta - V$ in $M$, where $V$ relates to $W$ (typically, $W = \tau V$ for some $\tau \in \R^+$). This is the case, for instance, of minimal hypersurfaces in ambient spaces with controlled sectional (or, more general, bi-Ricci) curvature \cite{cmmr_criticality}, or solitons for the Ricci and mean curvature flows \cite{rsol,msol, cmbook}, just to name a few. Picking a positive function $u \in C^2(M)$ satisfying 
\[
\Delta u + Vu \le 0,
\]
it is customary to consider the weighted Laplacian 
\[
\LL_{u} = u^{-2} \diver\left( u^2 \nabla \right),
\]
the reason being that $\xi = w/u$ solves
\begin{equation}\label{eq_doob}
\LL_u \xi \ge (V-W) \xi
\end{equation}
and thus, if $V \ge W$, $\xi$ is subharmonic for the weighted Laplacian $\LL_u$. Such procedure is called the \emph{Doob transform}, see \cite{car}. The operator $\LL_u$ is symmetric with respect to the weighted measure $u^2 \di V_g$, with $\di V_g$ the Riemannian volume of $M$, and to study solutions to \eqref{eq_doob} one is led to consider the weighted manifold $M_u \doteq (M^n,g,u^2 \di V_g)$. However, often it is not possible to control the geometry of $M_u$, for instance, its Bakry-Emery Ricci tensor or the volume growth of balls with respect to $u^2 \di V_g$. A notable exception was discovered in \cite{CMR}, see also \cite[Theorem 3.12]{cmmr_criticality} and \cite{car_mon_tew}. 

Our contribution is the next Lemma where, out of suitable powers of $w$ and $u$, we produce subsolutions of equations like \eqref{eq_doob} but \emph{still involving the unweighted Laplacian,} see item $(i)$. To the best of out knowledge, so far inequality \eqref{eq_magic} below has not  been observed, in this general form, in the literature, although it is somehow related to some computations in \cite[Step 4, Theorem 8.9]{prs_book}. Item $(i)$ is linked to an important convexity property in criticality theory first observed by Pinchover \cite{pinchover2}, which we include for later use in item $(ii)$. However, even though both $(i)$ and $(ii)$ deal with linear inequalities,  $(i)$ produces a subsolution from a sub- and a supersolution, while $(ii)$ only concerns supersolutions.

\begin{lemma}[\textbf{Key Lemma}]\label{prop_magic}
	Let $w,v,v_0 ,v_1$ be positive $C^2$ functions defined in an open set $\Omega \subset M$.
	\begin{itemize}
		\item[$(i)$] Assume that 
		\[
		\left\{ \begin{array}{l}
			\Delta w + W w \ge 0 \\[0.2cm]
			\Delta v + V v \le 0 
		\end{array}\right. \qquad \text{in } \, \Omega, 	
		\]
		for some $V,W \in C(\Omega)$. Then, for each $\delta \ge 0$ the function
		\[
		\xi_\delta \doteq w^{1+\delta} v^{-\delta}
		\]
		solves
		\begin{equation}\label{eq_magic}
			\Delta \xi_\delta \ge \big[\delta V - (1+\delta)W\big]\xi_\delta + \delta(1+\delta)\left| \nabla \log \left(\frac{w}{v}\right)\right|^2 \xi_\delta. 
		\end{equation}
		\item[$(ii)$] \cite[Theorem 3.1]{pinchover2} Assume that 
		\[
		\left\{ \begin{array}{l}
			\Delta v_0 + V_0 v_0 \le 0 \\[0.2cm]
			\Delta v_1 + V_1 v_1 \le 0 
		\end{array}\right. \qquad \text{in } \, \Omega, 	
		\]
		for some $V_0,V_1 \in C(\Omega)$. Then, for each $t \in [0,1]$, the function
		\[
		v_t \doteq v_0^{1-t}v_1^t
		\]
		solves
		\begin{equation}\label{eq_magic_super}
			\Delta v_t + V_t v_t \le - t(1-t) \left| \nabla \log \left(\frac{v_1}{v_0}\right)\right|^2 v_t \qquad \text{where } \ \ V_t \doteq (1-t)V_0 + tV_1.
		\end{equation}
	\end{itemize}
\end{lemma}

\begin{proof} 
	We first prove $(i)$. Write 
	\begin{equation}\label{eq_pri}
		\Delta \xi_\delta = \Delta\left( e^{\log \xi_\delta}\right) = \xi_\delta \left( \Delta \log \xi_\delta + |\nabla \log \xi_\delta|^2 \right).
	\end{equation}
	We thus compute
	\[
	\Delta \log \xi_\delta = (1+\delta) \Delta \log w - \delta \Delta \log v \ge \delta V - (1+\delta)W - (1+\delta)|\nabla \log w|^2 + \delta |\nabla \log v|^2
	\]
	and 
	\[
	|\nabla \log \xi_\delta|^2 = (1+\delta)^2|\nabla \log w|^2 + \delta^2|\nabla \log v|^2 - 2\delta(1+\delta) \langle \nabla \log w, \nabla \log v \rangle,
	\]
	where $\langle \,\, \rangle$ is the metric on $M$, to deduce
	\[
	\begin{array}{lcl}
		\Delta \log \xi_\delta + |\nabla \log \xi_\delta|^2 & \ge & \disp \delta V - (1+\delta)W + \delta(1+\delta)|\nabla \log w|^2 + \delta(1+\delta)|\nabla \log v|^2 \\[0.2cm] 
		& & \disp  - 2\delta(1+\delta) \langle \nabla \log w, \nabla \log v \rangle \\[0.2cm]
		& = & \disp \delta V - (1+\delta)W + \delta(1+\delta)\left|\nabla \log \left(\frac{w}{v}\right)\right|^2
	\end{array}
	\]
	which inserted into \eqref{eq_pri} leads to the desired inequality.\\ 
	The proof of $(ii)$ is analogous to $(i)$: we compute
	\[
	\begin{array}{lcl}
		\Delta \log v_t & = & (1-t) \Delta \log v_0 + t \Delta \log v_1 \le -V_t - (1-t)|\nabla \log v_0|^2 - t |\nabla \log v_1|^2\, , \\[0.2cm]
		|\nabla \log v_t|^2 & = & (1-t)^2 |\nabla \log v_0|^2 + t^2|\nabla \log v_1|^2 + 2t(1-t) \langle \nabla \log v_0, \nabla \log v_1 \rangle\, .
	\end{array}
	\]
	Summing up, 
	\[
	\Delta \log v_t + |\nabla \log v_t|^2 \le -V_t - t(1-t)\left|\nabla \log \left( \frac{v_1}{v_0}\right)\right|^2
	\]
	and from
	\[
	\Delta v_t = v_t \left( \Delta \log v_t + |\nabla \log v_t|^2 \right)
	\]
	the conclusion follows.
\end{proof}

\begin{remark}
	As shown in \cite{pinchover2}, $(ii)$ allows to prove that if 
	\[
	-\Delta - V_0 \ge 0, \quad -\Delta - V_1 \ge 0, \qquad V_t = (1-t)V_0 + tV_1,
	\]
	then $-\Delta - V_t$ is subcritical for $t \in (0,1)$ unless $V_0 \equiv V_1$. 
	To the best of our knowledge, the first application of this useful property to geometric analysis was the spectral splitting theorem in \cite{cmmr_criticality}. 
On the other hand, interestingly, $(ii)$ is a tool of common use for some problems in reaction-diffusion equations, see for instance \cite[Proposition 5.7]{bere_hamel} and the references therein. 
\end{remark}

Applying the above result to the gradient of a harmonic function we have the following:

\begin{proposition}\label{cor_magic_general}
	Let $(M^n,g)$ be a Riemannian manifold, and assume that
	\[
	\Ric \ge - \frac{n-1}{n-2} \tau V g, \qquad -\Delta - V \ge 0
	\]
	for some $\tau \in (0,\infty)$ and $V \in C(M)$. Let $\Omega \subset M$ be a domain and $0 < u,\ha \in C^2(\Omega)$ solve
	\[
	\Delta \ha = 0, \qquad \Delta u + Vu \le 0 \qquad \text{in } \, \Omega.
	\]
	
	Then, for each $\delta \ge 0$, the function $\xi_\delta \in C(\Omega)$ given by 
	\[
	\xi_\delta \doteq |\nabla \ha|^{\frac{n-2}{n-1}(1+\delta)} u^{-\delta}
	\]
	satisfies in the set $\{\nabla \ha \neq 0\}$ 
	\begin{equation}\label{eq_xidelta}
		\begin{array}{rcl}
			\Delta \xi_\delta & \ge & \disp (\delta - \tau - \delta\tau) V\xi_\delta + \delta(1+\delta)\xi_\delta \left| \nabla \log \left( \frac{|\nabla \ha|^{\frac{n-2}{n-1}}}{u}\right)\right|^2\\[0.5cm]
			\Delta \xi_\delta + V\xi_\delta & \ge & \disp (1+\delta)(1-\tau) V \xi_\delta + \delta(1+\delta)\xi_\delta \left| \nabla \log \left( \frac{|\nabla \ha|^{\frac{n-2}{n-1}}}{u}\right)\right|^2.
		\end{array}
	\end{equation}
	Moreover, $\xi_\delta \in H^1_\loc(\Omega)$ and the inequalities 
	\[
	\Delta \xi_\delta \ge \disp (\delta - \tau - \delta\tau) V\xi_\delta, \qquad \Delta \xi_\delta + V\xi_\delta \ge \disp (1+\delta)(1-\tau) V \xi_\delta 
	\]
	hold weakly in the entire $\Omega$.		
\end{proposition}

\begin{remark}\label{rem_vertical}
	Observe that $\xi_\delta$ may arrive ``vertically" on $\partial\{\xi_\delta > 0\}$ as the exponent of $|\nabla \ha|$ in the definition of $\xi_\delta$ may be smaller than $1$. Thus, the $H^1_\loc$ regularity of the extension by zero of $\xi_\delta$ is not obvious. Equation \eqref{eq_xidelta} will guarantee this property, see the Appendix.
\end{remark}

\begin{proof}
	By the Bochner formula, the refined Kato inequality and the assumed Ricci lower bound, $|\nabla \ha|$ solves
	\[
	\Delta |\nabla \ha| \ge \frac{1}{n-1} \frac{|\nabla |\nabla \ha||^2}{|\nabla \ha|} - \frac{n-1}{n-2}\tau V |\nabla \ha|  	
	\]
	pointwise in $\{\nabla h \neq 0\}$, whence $w \doteq |\nabla \ha|^{\frac{n-2}{n-1}}$ solves
	\[
	\Delta w + \tau V w \ge 0 \qquad \text{in } \, \{w>0\}\, . 
	\]
	The inequalities for $\Delta \xi_\delta$ and $\Delta \xi_\delta + V \xi_\delta$ are then immediate from $(i)$ in Lemma \ref{prop_magic}, applied with $W \doteq \tau V, \ v \doteq u$. The weak inequalities in all of $\Omega$ readily follow from Proposition \ref{prop_weakly} in the Appendix.  

\end{proof}

Let now $(M^n,g) \to \R^{n+1}$ be a  $2$-sided minimal hypersurface, and define the operator
\[
J_\tau u \doteq \Delta u + \frac{n-2}{n \tau}|A|^2 u,
\]
for $\tau \in (0,1]$. Note that the (negative of the) classical stability operator is obtained for the choice
\[
\tau = \frac{n-2}{n}.
\]
An application of Lemma \ref{prop_magic} to the function $|A|$ gives:

\begin{proposition}\label{cor_magic_minimal}
	Let $M^n \to \R^{n+1}$ be a minimal hypersurface, $\Omega \subset M$ be an open subset and $0 < u \in C^2(\Omega)$ solve
	\[
	J_\tau u  \le 0 \qquad \text{in } \, \Omega,
	\]
	for some $\tau \in (0,\infty)$. 
	
	Then, for each $\delta \ge 0$, the function $z_\delta \in C(\Omega)$ given by 
	\[
	z_\delta \doteq |A|^{\frac{n-2}{n}(1+\delta)} u^{-\delta}
	\]
	satisfies, in the set $\Omega\cap\{|A|>0\}$,
	\begin{equation}\label{eq_xidelta_A}
		\begin{array}{lcl}
			\Delta z_\delta & \ge & \disp \frac{n-2}{n\tau}(\delta - \tau -  \delta\tau)|A|^2 z_\delta + \delta(1+\delta)z_\delta \left| \nabla \log \left( \frac{|A|^{\frac{n-2}{n}}}{u}\right)\right|^2\,,\\[0.5cm]
			J_\tau z_\delta & \ge & \disp \frac{n-2}{n \tau}(1-\tau)(1+\delta)|A|^2 z_\delta + \delta(1+\delta)z_\delta \left| \nabla \log \left( \frac{|A|^{\frac{n-2}{n}}}{u}\right)\right|^2\, .
		\end{array}
	\end{equation}	
		
	Moreover, for each $t \in [0,1]$ the function $z_{\delta,t} \in C(\Omega)$ given by 
	\[
	z_{\delta,t} \doteq |A|^{\frac{n-2}{n}(1+\delta)} u^{-t\delta}
	\]
	satisfies, in the set $\Omega\cap\{|A|>0\}$,  
	\begin{equation}\label{eq_xidelta_A_2}
		\begin{array}{lcl}
			\Delta z_{\delta,t} & \ge & \disp \frac{n-2}{n\tau}(t\delta - \tau -  \delta\tau)|A|^2 z_{\delta,t}\, ,  \\[0.5cm]
			J_\tau z_{\delta,t} & \ge & \disp \frac{n-2}{n \tau}(t\delta - \tau -  \delta\tau + 1)|A|^2 z_{\delta,t}.
		\end{array}
	\end{equation}
	Finally, $z_\delta, z_{\delta,t} \in H^1_\loc(\Omega)$, and the inequalities \eqref{eq_xidelta_A_2} hold weakly in the entire $\Omega$. 		
\end{proposition}

\begin{proof}
We perform computations in the set $\Omega\cap\{|A|>0\}$. From  \eqref{eq_inesimons}, 
	we deduce that 
	\[
	w \doteq |A|^{\frac{n-2}{n}} \qquad \text{solves} \qquad \Delta w + \frac{n-2}{n}|A|^2 w \ge 0.
	\]
	We apply $(i)$ in Lemma \ref{prop_magic} with the choices
	\[
	V \doteq \frac{n-2}{n \tau}|A|^2, \qquad W \doteq \frac{n-2}{n}|A|^2, \qquad v = u  
	\]
	to deduce \eqref{eq_xidelta_A}. Observe that, for $t \in [0,1]$, the function $u^t$ solves
	\[
	\Delta u^t + t \frac{n-2}{n \tau}|A|^2 u^t \leq - t(1-t)|\nabla \log u|^2 u^t \le 0, 
	\]
	whence to prove \eqref{eq_xidelta_A_2} it is enough to apply $(i)$ in Lemma \ref{prop_magic} with the choices 
	\[
	V \doteq t\frac{n-2}{n \tau}|A|^2, \qquad W \doteq \frac{n-2}{n}|A|^2, \qquad v = u^t
	\]
	and to discard the gradient terms. The weak inequalities in \eqref{eq_xidelta_A} and \eqref{eq_xidelta_A_2} follow from Proposition \ref{prop_weakly} in the Appendix. 
\end{proof}

In the next corollary, we infer from Proposition \ref{cor_magic_minimal} a weighted integral inequality related to Schoen-Simon-Yau's inequality \eqref{eq_SSY}.

\begin{corollary}\label{cor_magic_minimal_2}
	Let $M^n \to \R^{n+1}$ be a minimal hypersurface, $\Omega \subset M$ be an open subset and $0 < u \in C^2(\Omega)$ solve
	\[
	J_\tau u  \le 0 \qquad \text{in } \, \Omega,
	\]
	for some $\tau \in (0,\infty)$. 
	
	Then, for each $\delta \ge 0$, $t \in (0,1]$ and $\psi \in \lip_c(\Omega)$ it holds
	\begin{equation}\label{beautiful_1}
		\begin{array}{ll}
			\disp \frac{n-2}{n \tau}(t\delta - \tau -  \delta\tau + 1) \int \left(|A|^{1+ \frac{n-2}{n}(1+\delta)} u^{-t\delta}\right)^2\psi^2 \le \int \left(|A|^{\frac{n-2}{n}(1+\delta)} u^{-t\delta} \right)^2 |\nabla \psi|^2\, .  
		\end{array}
	\end{equation}
	Moreover, if
	\[
	t\delta - \tau -  \delta\tau + 1 > 0,
	\]
	then
	\begin{equation}\label{beautiful_2}
		C_{\delta, \tau, t} \int \left(|A|^{1+ \frac{n-2}{n}(1+\delta)} u^{-t\delta}\right)^2 |\psi|^{2+ 2\frac{n-2}{n}(1+\delta)} \le \int u^{- 2t\delta} |\nabla \psi|^{2+ 2\frac{n-2}{n}(1+\delta)}
	\end{equation}
	for some constant $C_{\delta, \tau,t} >0$. 
\end{corollary}

\begin{proof}
	We apply \eqref{eq_id_basic_impo} with the choices $z=z_{\delta,t}$ and $V = \frac{n-2}{n\tau}|A|^2$ in Proposition \ref{cor_magic_minimal}. Inequality \eqref{beautiful_1} directly follows by the second in \eqref{eq_xidelta_A}. 
	Next, \eqref{beautiful_2} follows by considering the test function $\psi^{1+\frac{n-2}{n}(1+\delta)}$ in place of $\psi$ in \eqref{beautiful_1}, and applying H\"older's inequality as in the inequalities leading to \eqref{EstBer} above.
\end{proof}

\subsection{The gradient estimate}

Given a smooth open set $\UU \Subset M$, we say that a function $\ha \in C(M \backslash \UU)$ with $\ha>0$ on $\partial \UU$ is a \emph{minimal solution} to $\Delta \ha=0$ in $M \backslash \UU$ if there exists an exhaustion $\{M_j\}$ of $M$ by smooth, relatively compact domains satisfying 
\[
\UU \Subset M_j \Subset M_{j+1}, \qquad \bigcup_{j=1}^\infty M_j = M,
\]
and functions $\ha_j \in C^\infty(\overline{M}_j \backslash \UU)$ with the following properties:
	\[
	\left\{ \begin{array}{ll}
		\Delta \ha_j =0 & \quad \text{in } \, M_j \backslash \overline\UU, \\[0.2cm] 
		\ha_j >0 & \quad \text{on } \, \partial \UU, \\[0.2cm] 
		\ha_j = 0 & \quad \text{on } \, \partial M_j, \\[0.2cm]
		\ha_j \to \ha & \quad \text{locally uniformly in $M \backslash \UU$.}
	\end{array} \right.
	\]
Elliptic estimates (see e.g. the classical \cite{gt}) guarantee that $\ha \in C^\infty(M \backslash \UU)$ and solves $\Delta \ha = 0$ in $M \backslash \overline{\UU}$, and that $\ha_j \to \ha$ holds in $C^k_\loc(M \backslash \UU)$ for each $k$, in particular, up to $\partial \UU$. By the strong maximum principle we have $\ha_j>0$ in $M_j \backslash \UU$, and by comparison $\{\ha_j\}$ is pointwise increasing, whence $\ha>0$ in $M \backslash \UU$. The term ``minimal" comes from the following observation: if $v$ is any positive solution to $\Delta v \le 0$ in $M \backslash \overline{\UU}$ which is continuous and positive up to $\partial \UU$, then by comparing $v$ to $\ha_j$ and passing to limits one deduces that $\ha \le C v$ in $M \backslash \UU$, where $C = \max_{\partial \UU}(\ha/v)$. 

\begin{example}
	If we assume that $M$ is non-parabolic, by its very construction the minimal positive Green kernel of $M$ with a given pole $o$ is a minimal harmonic function in $M \backslash \UU$, for each smooth open set $\UU \Subset M$ containing $o$.  
\end{example}

Minimal harmonic functions enjoy the following integrability property: 

\begin{lemma}\label{lem_inte_h}	
	Let $\ha$ be a minimal harmonic function on $M \backslash \UU$. Then, for each $0 < \psi \in C(\R^+)$ satisfying 
	\[
	\int_{0}^1 \psi(s) \di s < \infty
	\]
	it holds	
	\[
	\int_{M \backslash \UU} \psi(\ha)|\nabla \ha|^2 < \infty. 
	\]
	In particular, 
	\[
	\int_{M \backslash \UU} \ha^{-a}|\nabla \ha|^2 < \infty \qquad \forall \, a \in (-\infty,1). 
	\]
\end{lemma}

\begin{proof}
Consider the sequences $\{M_j\}$ and $\{\ha_j\}$ in the definition of minimality for $\ha$. Define 
\[
\Psi : [0,\infty) \to (0,\infty), \qquad \Psi(t) = \int_0^t \psi(s) \di s, 
\]
and for $\eps>0$ integrate $\Delta \ha_j = 0$ in $M_j\backslash \UU$ against the Lipschitz function $[\Psi(\ha_j) - \Psi(\eps)]_+$ to deduce
\[
0 = \int_{\partial(M_j \backslash \UU)} [\Psi(\ha_j) - \Psi(\eps)]_+ \partial_\nu \ha_j - \int_{(M_j \backslash \UU) \cap \{\Psi(h_j) > \eps\}} \psi(\ha_j)|\nabla \ha_j|^2, 
\]	
where $\nu$ is the exterior normal to $M_j \backslash \UU$. Since $\ha_j = 0$ on $\partial M_j$ and $\Psi(0)=0$ thanks to the integrability of $\psi(s)$ at $s=0$, we get
\[
\int_{(M_j \backslash \UU)\cap \{\Psi(h_j) > \eps\}} \psi(\ha_j)|\nabla \ha_j|^2 = \int_{\partial \UU} [\Psi(\ha_j) - \Psi(\eps)]_+\partial_\nu \ha_j.
\]

Letting $\eps \to 0$ and using monotone convergence on the left hand side and dominated convergence on the right hand side, we deduce 
\[
\int_{M_j \backslash \UU} \psi(\ha_j)|\nabla \ha_j|^2 = \int_{\partial \UU} \Psi(\ha_j) \partial_\nu \ha_j.
\]
Letting $j \to \infty$, using Fatou's theorem on the left-hand side and Lebesgue theorem on the right-hand side, we conclude
\[
\int_{M \backslash \UU} \psi(\ha)|\nabla \ha|^2 \le \int_{\partial \UU} \Psi(\ha) \partial_\nu \ha < \infty,
\]
as claimed.
\end{proof}

We are ready to prove our main gradient estimate for minimal harmonic functions, which implies item {\it (i)} in Theorem \ref{teo_gradesti_intro}.

\begin{theorem}\label{teo_gradesti}
	Let $(M^n,g)$ be a complete, non-compact Riemannian manifold of dimension $n \ge 3$ satisfying
	\[
	\Ric \ge - \frac{n-1}{n-2} \tau V g, \qquad -\Delta - V \ \ \text{ has finite index,} 
	\]
	for some $0 \le V \in C(M)$ and 
	\begin{equation}\label{eq_cases_tau}
		\tau \in \left[0, \frac{3}{4}\right) \ \ \text{ if } \, n=3, \qquad \tau \in [0,1) \ \ \text{ if } \, n \ge 4.
	\end{equation}
	Let $\UU \Subset M$ be a smooth open set and assume that there exists $u \in C^2(M \backslash \UU)$ solving
	\[
	u > 0 \quad \text{in } \, M \backslash \UU, \qquad \Delta u + V u \le 0 \qquad \text{in } \, M\backslash \overline\UU.
	\]
	Let $\ha$ be a positive minimal solution to $\Delta \ha = 0$ in $M \backslash \UU$.
	
	Then,
	\begin{equation}\label{eq_gradesti_t}
		|\nabla \ha|^{\frac{n-2}{n-1}} \le C_\tau u^{\tau} \ha^{1-\tau}  \qquad \text{in } \, M \backslash \UU,
	\end{equation}
	where 
	\[
	C_\tau = \max_{\partial \UU} \left\{ |\nabla \ha|^{\frac{n-2}{n-1}} u^{-\tau} \ha^{\tau-1}\right\}.
	\]
\end{theorem}

\begin{proof}
	If $\ha$ is constant then the conclusion is obvious, so let us assume $\{|\nabla \ha|>0\} \neq \emptyset$. Define for convenience 
	\[
	w \doteq |\nabla \ha|^{\frac{n-2}{n-1}}. 
	\]
	Then, for each $\eps \in \R^+$ we have
	\begin{equation}\label{eq_niceeq}
		\Delta w + \tau V w \ge 0, \qquad \Delta u + Vu \le 0, \qquad \Delta (\ha+\eps) = 0
	\end{equation}
	pointwise in $\{|\nabla \ha| > 0\}$, where the first follows by combining Bochner inequality, the refined Kato inequality and our assumptions on $\Ric$.  
	
	We first provide a simpler argument, the one appearing in Section \ref{sec_bernstein}, to prove the theorem under the further conditions
	\begin{equation}\label{eq_extracond}
		V \in L^\infty(M), \qquad \tau < \frac{n-2}{n-1}, \qquad \ha(x) \to 0 \ \ \text{ as } \, x \to \infty.
	\end{equation}
	Proposition \ref{cor_magic_general} with the choice $\delta = \tau/(1-\tau)$ guarantees that the function
	\[
	\xi_1 \doteq w^{\frac{1}{1-\tau}} u^{-\frac{\tau}{1-\tau}},
	\]
	defined in $M \backslash \UU$, satisfies
	\begin{equation}\label{eq_magical_grad}
		\Delta \xi_1 \ge \disp  \frac{\tau}{(1-\tau)^2} \xi_1 \left| \nabla \log \left( \frac{|\nabla \ha|^{\frac{n-2}{n-1}}}{u}\right)\right|^2 \ge 0 \quad \text{pointwise in $\{|\nabla \ha|>0\}$}\, .
	\end{equation}
	Hence, $\Delta \xi_1 \ge 0$ holds weakly in the entire $M \backslash \overline{\UU}$ thanks to Proposition \ref{prop_weakly} in the Appendix.
	
	By the definition of $C_\tau$, we have
	\[
	\xi_1 \le C_\tau^{\frac{1}{1-\tau}} \ha \doteq C_\tau' \ha \qquad \text{on } \, \partial \UU.
	\]
	Furthermore, being $V \ge 0$ the function $u$ is superharmonic, whence 
	\begin{equation}\label{eq_minim}
	\ha \le C_\UU u \qquad \text{in } \, M \backslash \UU
	\end{equation}
	by the minimality of $\ha$, where we set $C_\UU = \max_{\partial \UU} \ha/u$. Since $V$ is bounded then $\Ric$ is bounded below by some constant. Hence, having fixed $\UU' \Subset M$ such that $\UU \Subset \UU'$, by Cheng-Yau's estimate (see \cite{ChengYau}) the bound $|\nabla \ha| \le C_1 \ha$ holds in $M \backslash \UU'$ for some constant $C_1 = C_1(\UU', \|V\|_\infty)>0$. Using this inequality first and then \eqref{eq_minim} we deduce 
	\[
	\xi_1(x) \le C_2 \ha(x)^{\frac{(n-2)}{(n-1)(1-\tau)}} u(x)^{-\frac{\tau}{1-\tau}} \le C_3 \ha(x)^{ \frac{1}{1-\tau} \left[\frac{n-2}{n-1}- \tau \right]} \to 0 \qquad \text{as } \, x \to \infty,
	\]
	where we used the last two conditions in \eqref{eq_extracond}. Finally, the function $\xi_1 - C_\tau'\ha$ is therefore subharmonic, non-positive on $\partial \UU$ and vanishes as $x \to \infty$, whence by comparison $\xi_1 - C_\tau'\ha \le 0$ in $M \backslash \UU$. Rearranging, we get the desired bound 
	\eqref{eq_gradesti_t}.
	
	We next prove the theorem in the stated generality. Let $\eps > 0$. From \eqref{eq_niceeq} and $(ii)$ in Lemma \ref{prop_magic}, for each $\theta,t \in [0,1]$ the functions
	\[
	u_\theta \doteq  (\ha+\eps)^{1-\theta}u^\theta, \qquad u_t \doteq  (\ha + \eps)^{1-t}u^t
	\]
	solve respectively
	\[
	\Delta u_\theta + \theta V u_\theta \le 0, \qquad \Delta u_t + tV u_t \le 0 \qquad \text{in } \, M \backslash \UU.
	\]
	Hence, by $(i)$ in Lemma \ref{prop_magic} applied with the choice $\delta = \frac{1}{n-2}$, the function
	\[
	\xi = w^{\frac{n-1}{n-2}}u_\theta^{- \frac{1}{n-2}} \qquad \text{solves} \qquad \Delta \xi \ge \left[\frac{\theta - (n-1)\tau}{n-2}\right] V \xi
	\]
	in the set 
	\[
	\Omega \doteq \{ |\nabla \ha|>0 \} \cap (M \backslash \overline{\UU}).
	\]
	Moreover, by \eqref{eq_doob} $z \doteq \xi/u_t$ solves
	\[
	\LL_{u_t} z \ge \left[\frac{\theta - (n-1)\tau + (n-2)t}{n-2}\right] V z \qquad \text{in } \, \Omega.
	\]
	where 
	\[
	\LL_{u_t} \phi = u_t^{-2} \diver \left( u_t^2 \nabla \phi \right).
	\]
	
Since $u_\theta\ge \ha^{1-\theta}u^{\theta}$ and $u_t\ge \ha^{1-t}u^{t}$, we have that 
$$
\max_{\partial \UU} z  \le a \doteq \max_{\partial U} \left( w^{\frac{n-1}{n-2}} \left[ \ha^{1-\theta}u^\theta \right]^{-\frac{1}{n-2}} \ha^{t-1}u^{-t}\right). 
$$
We note that the value of $a$ is independent of $\eps$ (indeed, it depends only on $n,h,u,\theta$ and $t$).

	We will choose $t,\theta\in[0,1]$ later to satisfy
	\begin{equation}\label{eq_nice_thetat}
		\theta -(n-1)\tau + (n-2)t = 0.
	\end{equation}
	Then, recalling the definition of $w$, 
	\[
	a = \max_{\partial U} \left( |\nabla \ha| \ha^{-\frac{n-1}{n-2}(1-\tau)} u^{- \frac{n-1}{n-2}\tau}\right).
	\]
	Moreover, the non-negative function $(z - a -\eps)_+$, considered as a function defined in $\Omega$, has compact support in $\Omega$ and satisfies
	\[
	\LL_{u_t}(z - a -\eps)_+ \ge 0 \qquad \text{weakly in } \, \Omega. 
	\]
	Since $(z - a -\eps)_+$ vanishes in a neighbourhood in $\Omega$ of $\partial \Omega$, extending $(z - a -\eps)_+$ with zero in $M \backslash \Omega$ yields a weak solution of the same inequality in the entire $M$. Moreover, 
	\[
	\int_M (z - a-\eps)_+^2 u_t^2 \le \int_{\Omega} z^2 u_t^2 = \int_{\Omega} w^{2 \frac{n-1}{n-2}} u_\theta^{- \frac{2}{n-2}} = \int_{\Omega} |\nabla \ha|^2 u_\theta^{- \frac{2}{n-2}}\, . 
	\]
	Using $\ha \le C_\UU u$ and $\eps \le \ha + \eps \le C$ in $M\backslash\UU$, we estimate
	\[
	u_\theta^{- \frac{2}{n-2}} = u^{-\frac{2\theta}{n-2}}(\ha+\eps)^{- \frac{2(1-\theta)}{n-2}} \le C_\eps'u^{-\frac{2\theta}{n-2}} \le C_\eps'' \ha^{-\frac{2\theta}{n-2}} \qquad \text{ in } \, M\backslash\overline{\UU}\supset\Omega. 
	\]
	Hence, we conclude
	\[
	\int_M (z - a-\eps)_+^2 u_t^2 \le C_\eps''\int_{M\backslash\UU}|\nabla \ha|^2 \ha^{- \frac{2\theta}{n-2}}.
	\]
	
	We therefore examine two cases:
	\begin{itemize}
		\item[(i)] If $\tau \in \left[ 0, \frac{n-2}{2}\right) \cap [0,1)$, then we set $t = \theta = \tau\in[0,1)$. Then, \eqref{eq_nice_thetat} is satisfied and $\frac{2\theta}{n-2} < 1$, whence by Lemma \ref{lem_inte_h}
		\begin{equation}\label{eq_inte_gradh}
			\int_{M\backslash\UU} |\nabla \ha|^2 \ha^{- \frac{2\theta}{n-2}} < \infty.
		\end{equation}
		\item[(ii)] If $\tau \in \left[ \frac{n-2}{n-1}, \frac{3}{2} \frac{n-2}{n-1} \right) \cap [0,1)$, then we set $t = 1$, $\theta = (n-1)\tau - (n-2)\in[0,1)$. Again, \eqref{eq_nice_thetat} is satisfied and 
		\[
		\frac{2\theta}{n-2} = 2\tau \frac{n-1}{n-2} - 2 < 1, 
		\]
		thus \eqref{eq_inte_gradh} holds as well.
	\end{itemize}	
	Note that the intervals in $(i),(ii)$ for $\tau$ exhaust all the ranges in \eqref{eq_cases_tau}. 
	
	Finally, from either $(i)$ or $(ii)$ we deduce that
	\[
	\int_M (z - a-\eps)_+^2 u_t^2 \le C_\eps''\int_{M\backslash\UU} |\nabla \ha|^2 \ha^{- \frac{2\theta}{n-2}} < \infty.
	\]
	We can therefore apply Yau's theorem  in \cite{yau} for non-negative subharmonic functions on complete, non-compact manifolds (its extension to weakly $\LL_{u_t}$-subharmonic ones can be found in \cite[Theorem 4.1]{prs_book}) to deduce that $(z - a -\eps)_+ \equiv 0$, so
	\[
	z \le a + \eps \qquad \text{in } \, M \backslash \UU, 
	\]
	namely, 
	\[
	|\nabla \ha| \le (a + \eps) u_t u_\theta^{\frac{1}{n-2}}.
	\]
	Recalling the definitions of $u_t$ and $u_\theta$, 
	\[
	|\nabla \ha| \le (a + \eps) u^{\frac{\theta}{n-2} + t} (\ha + \eps)^{\frac{1-\theta}{n-2} + 1-t}
	\]
	and using that $\theta + (n-2)t = (n-1)\tau$ we get
	\[
	|\nabla \ha| \le (a + \eps) u^{\frac{n-1}{n-2}\tau} (\ha + \eps)^{\frac{n-1}{n-2}(1-\tau)}.
	\]
	Letting $\eps \to 0$ and recalling the value of $a$ (which is independent of $\eps$) we get the desired estimate for $|\nabla \ha|$. 
\end{proof}


\begin{remark}
	It seems reasonable to presume that \eqref{eq_gradesti_t} holds dimensionless in the whole interval $\tau \in [0,1]$. Note that, in the limit $\tau = 1$ the estimate reads as 
	\[
	|\nabla \ha|^{\frac{n-2}{n-1}} \le C u^{\tau} \ha^{1-\tau} = C u.
	\]
\end{remark}

Next, we examine the case where $\ha$ is the minimal positive Green kernel $\gr$ of $-\Delta$ with pole at a given origin $o\in M$, and prove item {\it (ii)} in Theorem \ref{teo_gradesti_intro}, which we restate for the reader's convenience. As stated in the Introduction, the result generalizes Colding's \cite[Thm. 3.1]{colding} to manifold with possibly negative Ricci lower bounds. To estimate $|\nabla \gr|$ in  $M \backslash \{o\}$, we shall use as a comparison function the minimal positive Green kernel of $-\Delta - V$, whose existence is equivalent to require that $-\Delta - V$ is subcritical. Since $V \ge 0$, by \cite[Theorem 2.5]{murata} the subcriticality of $-\Delta - V$ implies that of $-\Delta$, whence the existence of $\gr$.

\begin{theorem}[\textbf{Gradient estimate for the Green kernel}]\label{teo_gradesti_kernel}
	Let $(M^n,g)$ be a complete Riemannian manifold of dimension $n \ge 3$ satisfying
	\[
	\Ric \ge - \frac{n-1}{n-2} \tau V g, \qquad -\Delta - V \ \ \text{ is subcritical,} 
	\]
	for some $0 \le V \in C^{0,\alpha}_\loc(M)$ and 
	\[
	\tau \in \left[0, \frac{3}{4}\right) \ \ \text{ if } \, n=3, \qquad \tau \in [0,1) \ \ \text{ if } \, n \ge 4.
	\]	
	Let $\gr$ be the minimal positive Green kernel of $-\Delta$ on $M$ with pole at $o$, and let $u$ be a positive kernel of $-\Delta - V$ with the same pole. 
	
	Then,                                                                                    
	\begin{equation}\label{eq_gradesti_kernel}
		|\nabla \gr|^{\frac{n-2}{n-1}} \le (n-2) |\mathbb{S}^{n-1}|^{\frac{1}{n-1}} u^{\tau} \gr^{1-\tau} \qquad \text{in } \, M \backslash \{o\}.
	\end{equation}
Moreover, equality holds for some $x_0 \in M \backslash \{o\}$ if and only if $V \equiv 0$ and $M$ is isometric to~$\R^n$. 
\end{theorem}

\begin{proof}
	Observe first that $M$ is non-compact because $-\Delta - V$, hence $-\Delta$, is subcritical. We can therefore apply Theorem \ref{teo_gradesti} to $\gr$ on $M \backslash B_\eps(o)$ for each fixed $\eps>0$ to get
	\[
|\nabla \gr|^{\frac{n-2}{n-1}} \le c_{\eps,\tau} u^{\tau} \gr^{1-\tau} \qquad \text{in } \, M \backslash B_\eps(o), \qquad c_{\eps,\tau} = \max_{\partial B_\eps(o)} \left\{ |\nabla \gr|^{\frac{n-2}{n-1}} u^{-\tau} \gr^{\tau-1}\right\}.
	\]
	By the asymptotic behaviour of singular solutions (see \cite[\S8,18,19]{miranda}, or \cite[Theorem 2.1]{bmrsx})
	we have
	\[
\gr(x) = u(x) \bigl(1+o(1)\bigr) = \frac{\di(x,o)^{2-n}}{(n-2)|\mathbb{S}^{n-1}|} \bigl(1+o(1)\bigr) \quad \text{ as } x \to o\, , 
\]
and 
\[
|\nabla \gr| = \frac{\di(x,o)^{1-n}}{|\mathbb{S}^{n-1}|} \bigl(1+o(1)\bigr)  \quad \text{ as } x \to o\, ;
\] 
whence $c_{\eps, \tau} \to (n-2) |\mathbb{S}^{n-1}|^{\frac{1}{n-1}}$ as $\eps \to 0$, concluding \eqref{eq_gradesti_kernel}. 
	
	Assume that equality in \eqref{eq_gradesti_t} holds at some point $x_0 \in M \backslash \{o\}$. In the proof of Theorem \ref{teo_gradesti} we observed, in \eqref{eq_magical_grad} and right after it, that the function
	\[
	\xi_1 \doteq |\nabla \gr|^{\frac{n-2}{(n-1)(1-\tau)}} u^{-\frac{\tau}{1-\tau}}
	\]
	satisfies, pointwise in $\{|\nabla\gr|>0\}=\{\xi_1 > 0\}$,
	\begin{equation}\label{xi1_point}
		\Delta \xi_1 \ge \disp  \frac{\tau}{(1-\tau)^2} \xi_1 \left| \nabla \log \left( \frac{|\nabla \gr|^{\frac{n-2}{n-1}}}{u}\right)\right|^2 \ge 0.
	\end{equation}
Hence, equality in \eqref{eq_gradesti_kernel} means that the subharmonic function 
	\[
	\xi_1 - \left[(n-2) |\mathbb{S}^{n-1}|^{\frac{1}{n-1}}\right]^{\frac{1}{1-\tau}} \gr
	\]
	attains a maximum at $x_0 \in \{|\nabla \gr|>0\}$ with value $0$. By the maximum principle it is therefore identically zero, so $\xi_1$ is harmonic. From this and \eqref{xi1_point} we deduce that
	\begin{equation}\label{eq_nablaG_eq}
	|\nabla \gr|^{\frac{n-2}{n-1}}/u \qquad \text{is a constant $C$ in $\{|\nabla \gr|>0\}$}.
	\end{equation}
	
	Consequently, $|\nabla \gr|> 0$ in $M \backslash \{o\}$, and in view of the asymptotic behaviour of $|\nabla \gr|$ and $u$ as $x \to o$ we have  
	\[
	C = (n-2) |\mathbb{S}^{n-1}|^\frac{1}{n-1}.
	\]
	We compute in $M \backslash \{o\}$:
	\begin{equation}\label{eq_bella!}
	0 = \Delta u + Vu = C^{-1} \left( \Delta |\nabla \gr|^{\frac{n-2}{n-1}} + V|\nabla \gr|^{\frac{n-2}{n-1}} \right).
	\end{equation}
	On the other hand, by the Bochner formula and the refined Kato inequality,
	\[
	\Delta |\nabla \gr|^{\frac{n-2}{n-1}} \ge - \frac{n-2}{n-1}|\nabla \gr|^{\frac{n-2}{n-1} - 2}\Ric(\nabla \gr,\nabla \gr) \ge - \tau V |\nabla \gr|^{\frac{n-2}{n-1}}.
	\]
Once plugged into \eqref{eq_bella!}, this yields 
	\[
	0 \ge (1-\tau)V |\nabla \gr|^{\frac{n-2}{n-1}}\ge 0.
	\]
	Since $\tau < 1$ we therefore conclude that $V \equiv 0$, so $\Ric \ge 0$ and $u = \gr$ by uniqueness of the minimal positive Green kernel. 
	
	Finally, the identity \eqref{eq_nablaG_eq} becomes
	\begin{equation}\label{=_grad}
	|\nabla \gr|^{\frac{n-2}{n-1}} = (n-2) |\mathbb{S}^{n-1}|^{\frac{1}{n-1}} \gr \qquad \text{in } \, M \backslash \{o\},
	\end{equation}
and thus, following \cite{colding}, simple computations using only \eqref{=_grad} and $\Delta\gr=0$ show that the function
	\[
	b = \big[ (n-2)|\mathbb{S}^{n-1}|\big]^{\frac{1}{2-n}} \gr^{\frac{1}{2-n}}
	\]
	solves
	\[
	|\nabla b| = 1, \qquad \Delta b = \frac{n-1}{b}\,. 
	\]
	In particular, since $\lim_{x \to o}b(x) = 0$ while $b>0$ on $\R^+$, $b$ is the distance function $r$ from $o$. From $b \in C^\infty(M \backslash \{o\})$ we deduce that $\cut(o) = \emptyset$, and the identity $\Delta r = (n-1)/r$ combined with $\Ric \ge 0$ imply rigidity in the Laplacian comparison theorem, whence $M = \R^n$.
\end{proof}
\begin{remark}
	The rigidity part in the above theorem characterized Euclidean space. As a matter of fact, Colding's gradient estimate was later extended in \cite{mrs} to manifolds whose Ricci curvature satisfies
	\[
	\Ric \ge -(n-1){\rm H}(r), \qquad r(x) = {\rm dist}(x,o)
	\]
	for some $0 \le {\rm H} \in ([0,\infty))$ non-increasing, with rigidity characterizing the rotationally symmetric model space $M_{\rm H}$ of radial sectional curvature ${\rm H}$. It may be interesting to see whether the results in \cite{mrs} allow for extensions to spectral Ricci lower bounds.
\end{remark}
We conclude the section by establishing the following corollary for complete, $2$-sided minimal hypersurfaces $(M^n,g) \to N^{n+1}$ into an ambient manifold with non-negative sectional curvature. In this case, Gauss equation and the refined Kato inequality still imply that 
\begin{equation}\label{ineq_Ricci_min}
\Ric \ge - \frac{n-1}{n}|A|^2g,
\end{equation}
see \cite{cmmr_criticality}, and stability corresponds to the non-negativity of the Jacobi operator 
\[
-\Delta - \big[|A|^2 + \overline{\Ric}(\nu,\nu)\big].
\]
As explained at the end of Section \ref{sec_bernstein}, if $M$ is parabolic, from the existence of a positive solution to 
\[
\Delta u + \big[|A|^2 + \overline{\Ric}(\nu,\nu)\big] u \le 0 \qquad \text{in } \, M
\]
one readily gets $|A| \equiv 0$ and $\overline{\Ric}(\nu,\nu) \equiv 0$ on $M$. Therefore, we focus on non-parabolic manifolds.

\begin{corollary}\label{cor_gradesti_minimal}
	Let $M^n \to (N^{n+1},\bar g)$ be a $2$-sided, complete minimal hypersurface in an ambient manifold with $\overline{\Sec} \ge 0$. Assume that $M$ is non-parabolic and and let $\gr$ be the minimal positive Green kernel of $-\Delta$ centered at some origin $o \in M$.  
	\begin{itemize}
	\item[(i)] If $M$ has finite index, then for each smooth open set $\UU \Subset M$ for which there exists a solution $0 < u \in C^2(M \backslash \UU)$ to 
	\[
	\Delta u + \big[|A|^2 + \overline{\Ric}(\nu,\nu)\big] u \le 0 \qquad \text{in } \, M \backslash \overline{\UU},
	\]
	it holds
	\[
	|\nabla \gr|^{\frac{n-2}{n-1}} \le C u^{\frac{n-2}{n}} \gr^{\frac{2}{n}}, \qquad \text{where } \, C = \max_{\partial \UU} \left\{ |\nabla \gr|^{\frac{n-2}{n-1}} u^{-\frac{n-2}{n}} \gr^{- \frac{2}{n}} \right\}.
	\]
	\item[(ii)] If $M$ is stable and the Jacobi operator $ -\Delta - \big[|A|^2 + \overline{\Ric}(\nu,\nu)\big]$ is subcritical, letting $u$ be a positive Green kernel of the Jacobi operator with pole $o$, it holds
	\begin{equation}\label{eq_gradesti_kernelb}
		|\nabla \gr|^{\frac{n-2}{n-1}} \le (n-2) |\mathbb{S}^{n-1}|^{\frac{1}{n-1}} u^{\frac{n-2}{n}} \gr^{\frac{2}{n}} \qquad \text{in } \, M \backslash \{o\}.
	\end{equation}
	Moreover, equality holds for some $x_0 \in M \backslash \{o\}$ if and only if $M$ is totally geodesic and isometric to $\R^n$, and $\overline{\Ric}(\nu,\nu) \equiv 0$ along $M$.
	\end{itemize}
\end{corollary}

\begin{proof}
	In our assumptions inequality \eqref{ineq_Ricci_min} holds. It is therefore enough to apply Theorem \ref{teo_gradesti} in case \textit{(i)} and Theorem \ref{teo_gradesti_kernel} in case \textit{(ii)} with the choices 
	\[
	\tau = \frac{n-2}{n} \ \ \text{ (which satisfies \eqref{eq_cases_tau}),} \qquad V = |A|^2 + \overline{\Ric}(\nu,\nu) 
	\]
	to conclude.
\end{proof}

\section{Appendix}

In this short appendix we prove the following technical proposition,  which has been used on various occasions throughout the rest of the article. As pointed out in Remark \ref{rem_vertical}, the subtlety of the result is that the subsolution $f$ below is not assumed a priori to be in $H^1_\loc$ when extended by zero outside $\{f>0\}$. In fact, this is the main issue to be proved. 

\begin{proposition}\label{prop_weakly}
	Let $\Omega \subset M$ be an open subset of a Riemannian manifold, and assume that $f \in C(\Omega) \cap H^1_\loc(\{f>0\})$ solves
	\[
	\Delta f + P f \ge 0 \qquad \text{weakly in } \, \{f>0\},
	\]
	for some $P \in L^\infty_\loc(\Omega)$. 
	
	Then, the function $f_+$ belongs to $H^1_\loc(\Omega)$ and satisfies $\Delta f_+ + P f_+ \ge 0$ weakly in $\Omega$.
\end{proposition}
\begin{proof}[Proof of Proposition \ref{prop_weakly}] For $\eps>0$ define $f_\eps = \max\{ f-\eps,0\}$ in $\{f>0\}$. Recall that $f\in C(\Omega)$. Then, by the Brezis-Kato inequality (see \cite[Proposition 4.1]{pvv} for its proof in a manifold setting) $f_\eps \in H^1_\loc(\{f>0\})$ and the following holds in the distributional (and weak) sense in $\{f>0\}$:
 \[
 \Delta f_\eps \ge - (Pf_\eps + P\eps)\mathbb{1}_{\{f > \eps\}} \ge - (Pf_\eps + |P|\eps)\mathbb{1}_{\{f > \eps\}}. 
 \]
 Moreover, since $f_{\eps}$ vanishes in a neighbourhood of $\partial \{f>0\}$, its extension with zero (still called $f_\eps$) belongs to $H^1_\loc(\Omega)$ and satisfies
 \begin{equation}\label{eq_weakineq}
 \Delta f_\eps \ge - (Pf_\eps + |P|\eps)\mathbb{1}_{\{f > \eps\}}  
 \end{equation}
in the weak sense in the entire $\Omega$. 

We next integrate \eqref{eq_weakineq} against $\varphi^2 f_\eps$, with $\varphi \in C^1_c(\Omega)$, and apply the Cauchy-Schwarz and Young inequalities to the term with $\langle \nabla f_\eps, \nabla \varphi \rangle$ to deduce
\[
\frac{1}{2} \int \varphi^2 |\nabla f_\eps|^2 \le 2 \int f_\eps^2|\nabla \varphi|^2 + \int f_\eps(f_\eps + \eps)\varphi^2 |P|.
\]
Since $0\leq f_\eps\leq f\in C(\Omega)$, then $\{f_\eps\}$ satisfies uniform $H^1$ bounds in any given compact subset of $\Omega$. Moreover, from $f_\eps \to f$ in $L^2_\loc(\Omega)$, by \cite[Lemma 6.2 page 16]{friedman} we have $f \in H^1_\loc(\Omega)$ and $\nabla f_\eps \rightharpoonup \nabla f$ weakly in each compact subset. Letting $\eps \to 0$ in \eqref{eq_weakineq}, we deduce that $f$ satisfies the desired weak inequality.
\end{proof}

\section*{Acknowledgements}
This work was initiated while Alberto Roncoroni was visiting the Department of Mathematics of the Universitat Politècnica de Catalunya, whose hospitality is gratefully acknowledged. The authors would like to thank J.-M. Roquejoffre for useful discussions and for pointing out reference \cite{bere_hamel}.

\vspace{0.4cm}

\noindent \textbf{Conflict of Interest.} The authors have no conflict of interest.



\

\Addresses

\end{document}